\documentclass{article}
\usepackage{amssymb}   
\usepackage{amsmath}
\usepackage{amsthm}
\usepackage{mathabx}
\usepackage{enumerate}
\usepackage{xcolor}
\usepackage{changepage}

\usepackage{hyperref}

\theoremstyle{plain}

\newtheorem{Cor}{Corollary}

\newtheorem{Thm}{Theorem}

\newcommand*{\Ann}{\ensuremath{\mathrm{Ann\,}}}
\newcommand*{\Hom}{\ensuremath{\mathrm{Hom\,}}}
\newcommand*{\R}{\ensuremath{\mathbb{R}}}

\newcommand*{\Z}{\ensuremath{\mathbb{Z}}}
\newcommand*{\C}{\ensuremath{\mathbb{C}}}
\newcommand*{\N}{\ensuremath{\mathbb{N}}}

\begin{document}
	
	\date{}
	
	\author{
		L\'aszl\'o Sz\'ekelyhidi\\
		{\small\it Institute of Mathematics, University of Debrecen,}\\
		{\small\rm e-mail: \tt szekely@science.unideb.hu,}
		%	Kedumetse Vati\\
		%	{\small\it Institute of ???, University of ???,}\\
		%	{\small\rm e-mail: \tt ???,}
	}
	
	\title{On The Principal Ideal Theorem}

	\maketitle
	
	\begin{abstract}
		
			In our recent work, we solved a long-standing problem originating from Laurent Schwartz’s theory of mean-periodic functions by fully characterizing those locally compact abelian groups that admit spectral synthesis. The key tool is our localization technique, which we now apply to provide a conceptually transparent proof for the generalization of a fundamental result in spectral synthesis: the Principal Ideal Theorem on any locally compact abelian group.
		
	\end{abstract}
	
	\footnotetext[1]{The research was supported by the  the
		Hungarian National Foundation for Scientific Research (OTKA),
		Grant No.\ K-134191.}\footnotetext[2]{Keywords and phrases:
		variety, spectral synthesis}\footnotetext[3]{AMS (2000) Subject Classification: 43A45, 22D99}
	\vskip1cm
	
\section{Introduction}

The following result, established by B.~Malgrange in 1954 (see \cite{MR0060113}; see also \cite{MR2680008}), is commonly referred to as the \emph{Principal Ideal Theorem}:

\begin{Thm}
	For any nonzero linear partial differential operator $P(D)$ on $\R^n$, 
	the exponential monomials solving that solve the equation $P(D)f = 0$ linearly span a dense subspace of the solution space.
\end{Thm}

A related result was proved in \cite{MR0086990} (see also \cite{MR0076303}):

\begin{Thm}
	Given a compactly supported complex Borel measure $\mu$ on $\R^n$, the exponential monomials solving the convolution equation $\mu * f = 0$ span a dense subspace of the space of all continuous solutions.
\end{Thm}

Obviously, the same problem can be considered on any locally compact abelian group as the concept of {\it exponential monomial} is well-defined (see e.g. \cite{MR2680008, MR3185617}). More generally, if a system of convolution equations of the form
\begin{equation}\label{conveq}
	\mu*f=0\hskip2cm \mu\in \Gamma
\end{equation}
is given, where $\Gamma$ is a nonempty set of compactly supported measures on the locally compact abelian group $G$, then we say that this system is {\it synthesizable}, if the exponential polynomial solutions span a dense subspace in the solution space. Density is meant in the compact-open topology of the space $\mathcal C(G)$ of all continuous complex valued functions on $G$. The solution space of an equation system of the form \eqref{conveq}  is called a {\it variety} and it can be characterized as a closed translation invariant linear subspace in $\mathcal C(G)$. It turns out that synthesizability of the variety $V$ can be characterized in terms of its {\it annihilator} $\Ann V$, which is the set of all measures $\mu$ satisfying \eqref{conveq} for each $f$ in $V$. Indeed, $\Ann V$ is a closed ideal, and $V$ is synthesizable if and only if $\Ann V$ coincides with the intersection of all ideals of the form $M^k\supseteq \Ann V$, where $M$ is a maximal ideal whose residue algebra is the complex field $\C$, and $\Ann M^k$ is finite dimensional. For the details see \cite{MR2680008}, \cite[Theorem 17]{MR3502634}. 

In our recent work, we introduced the concept of localization of ideals in the Fourier algebra of a locally compact abelian group (see \cite{Sze23}). It turns out that the localizability of an ideal is equivalent to the synthesizability of the corresponding variety in the function space. Using this observation, we characterized those locally compact abelian groups admitting spectral synthesis. In the present paper, we apply the localization technique to provide a straightforward and transparent proof of the Principal Ideal Theorem for arbitrary locally compact abelian groups.

\section{Differential operators on the Fourier algebra}

Let $G$ be a locally compact abelian group, and let $\mathcal{M}_c(G)$ denote the \emph{measure algebra} of $G$, that is, the space of all compactly supported complex Borel measures on $G$, equipped with linear operations and the weak$^*$ topology. Furthermore, let $\mathcal{A}(G)$ denote the \emph{Fourier algebra} of $G$, consisting of all Fourier transforms of measures in $\mathcal{M}_c(G)$. Recall that the Fourier transform of a measure $\mu$ in $\mathcal{M}_c(G)$ is defined by
\[
\widehat{\mu}(m) = \int \widecheck{\mu}(x)\, d\mu(x) = \int \mu(-x)\, d\mu(x),
\]
whenever $m$ is an \emph{exponential} on $G$, that is, a continuous homomorphism from $G$ into the multiplicative group of nonzero complex numbers.

It is well known that the Fourier transform is a bijective homomorphism between $\mathcal{M}_c(G)$ and $\mathcal{A}(G)$. We equip $\mathcal{A}(G)$ with the topology induced by this bijection, turning it into a topological algebra that can be identified with $\mathcal{M}_c(G)$. Using this identification for each measure $\mu$ and each set $I$ in $\mathcal M_c(G)$ we use the notation $\widehat{\mu}$ and $\widehat{I}$ for the corresponding element and subset in $\mathcal A(G)$. Using the annihilator-variety correspondence, which is a one-to-one mapping between the closed ideals of the measure algebra $\mathcal M_c(G)$ and the varieties in $\mathcal C(G)$, we have a one-to-one mapping between the varieties in $\mathcal C(G)$ and closed ideals in the Fourier algebra $\mathcal A(G)$. Somewhat loosely, we say that a closed ideal in the Fourier algebra is synthesizable, if the corresponding variety in the function space is synthesizable.

Since the functions in $\mathcal{A}(G)$ are defined on the set $\widetilde{G}$ of all exponentials, it is natural to introduce a topology on this set. Clearly, $\widetilde{G}$ is a commutative group under pointwise multiplication. In fact, $\widetilde{G}$ is isomorphic to $\widehat{G} \times \Hom(G, \R)$, where $\widehat{G}$ is the Pontrjagin dual of $G$, and $\Hom(G, \R)$ is the space of all continuous real-valued homomorphisms on $G$.

Equipping $\Hom(G, \R)$ with the compact-open topology makes it a topological group, although it may not be locally compact in general. The natural topology on $\widetilde{G}$ is then the product topology. If $G$ is compactly generated, then $\Hom(G, \R)$ is a finite-dimensional vector space, and $\widetilde{G}$ is locally compact. The elements of $\Hom(G, \R)$ are called \emph{additive functions} or \emph{real characters}.
%\vskip.2cm

We define differential operators on the Fourier algebra using derivations. A continuous linear operator $D: \mathcal{A}(G) \to \mathcal{A}(G)$ is called a \emph{first-order derivation} if it satisfies
\[
D(\widehat{\mu} \cdot \widehat{\nu}) = D(\widehat{\mu}) \cdot \widehat{\nu} + \widehat{\mu} \cdot D(\widehat{\nu})
\]
for all $\widehat{\mu}, \widehat{\nu}$ in $\mathcal{A}(G)$. We call $D$ a \emph{first-order polynomial derivation} on $\mathcal{A}(G)$ if and only if there exists an additive function $a$ such that
\[
D\widehat{\mu}(m) = \int \widecheck{a}(x)\, \widecheck{m}(x)\, d\mu(x),
\]
for all $\widehat{\mu}$ in $\mathcal{A}(G)$. It is easy to verify that this formula indeed defines a first-order derivation.

In \cite{Sze23}, we proved that all first-order derivations generate a commutative algebra $\mathcal{P}(G)$, called the \emph{algebra of polynomial derivations} or \emph{polynomial differential operators}, which can be written in the form $D = P(D_1, D_2, \dots, D_k)$, where $P$ is a complex polynomial in $k$ variables. If $P$ is of degree $n$ and $D_1, D_2, \dots, D_k$ are linearly independent, then $D$ is called a derivation of order $n$. The identity operator is considered a derivation of order zero.
%\vskip.2cm

Differential operators can also be defined analytically. A continuous homomorphism $\gamma: \C \to \widetilde{G}$ is called a \emph{one-parameter group} in $\widetilde{G}$. For example, for each additive function $a$, define
\[
\gamma_a(z)(g) = \exp(za(g)).
\]
Then $\gamma_a$ defines a one-parameter group in $\widetilde{G}$.
%\vskip.2cm

Let $\gamma$ be a one-parameter group in $\widetilde{G}$ and let $m$ be an exponential. A function $f: \widetilde{G} \to \C$ is said to be \emph{differentiable along $\gamma$ at $m$} if the following finite limit exists:
\[
\partial_\gamma f(m) = \lim_{z \to 0} \lim_{w \to 0} \frac{1}{w} \left( f(m \cdot \gamma(z+w)) - f(m \cdot \gamma(z)) \right).
\]
If $\gamma = \gamma_a$, we write $\partial_a f(m)$ for $\partial_\gamma f(m)$. It is easy to show (see e.g. \cite[Theorem 11]{MR3481109}) that every one-parameter group in $\widetilde{G}$ is of the form $\gamma_a$ for some additive function $a$.
%\vskip.2cm

If $\partial_a f(m)$ exists for all $m$, then we can define the function $m \mapsto \partial_a f(m)$. If this function is differentiable along the one-parameter group $\gamma_b$ corresponding to another additive function $b$, then we can define the second derivative $\partial_b \partial_a f$, and so on. This leads to higher-order derivatives of the form $\partial_{a_1}^{\alpha_1} \partial_{a_2}^{\alpha_2} \dots \partial_{a_k}^{\alpha_k} f$, corresponding to additive functions.

If such derivatives exist for arbitrary choices of additive functions $a_1, a_2, \dots, a_k$, then we say that $f$ belongs to the class $C^\infty$. In \cite[Theorem 12]{Sze23}, we proved that every Fourier transform is a $C^\infty$ function, and
\begin{equation}\label{Foudiff}
	P(\partial_{a_1}, \partial_{a_2}, \dots, \partial_{a_k})\, \widehat{\mu}(m) =
	\int P(\widecheck{a}_1, \widecheck{a}_2, \dots, \widecheck{a}_k)\, \widecheck{m}(x)\, d\mu(x)
\end{equation}
holds whenever $a_1, a_2, \dots, a_k$ are additive functions, $P$ is a complex polynomial in $k$ variables, $\mu \in \mathcal{M}_c(G)$, and $m$ is an exponential on $G$.
%\vskip.2cm

These considerations lead to the following characterization:

\begin{Thm}
	Given additive functions $a_1, a_2, \dots, a_k$ on $G$ and a complex polynomial $P$ in $k$ variables, the mapping $P(\partial_{a_1}, \partial_{a_2}, \dots, \partial_{a_k})$ defined on the Fourier algebra by equation \eqref{Foudiff} is a polynomial derivation on $\mathcal{A}(G)$. Conversely, every polynomial derivation on $\mathcal{A}(G)$ is of the form \eqref{Foudiff}.
\end{Thm}

\begin{proof}
	The first part of the statement follows immediately from the fact that $\partial_a$ is a derivation for each additive function $a$. For the second part, assume that $D$ is a first-order derivation on $\mathcal{A}(G)$. As shown above, $D$ has the form
	\[
	D\widehat{\mu}(m) = \int \widecheck{a}(x)\, \widecheck{m}(x)\, d\mu(x).
	\]
	Comparing this formula with \eqref{Foudiff}, we conclude that $D = \partial_a$.
\end{proof}

\section{Localization}

Let $\widehat{I}$ be an ideal in the Fourier algebra $\mathcal{A}(G)$, and let $\mathcal{P}$ be a set of polynomial differential operators, with $m$ an exponential. We say that $\mathcal{P}$ \emph{annihilates $\widehat{I}$ at $m$} if
\[
\partial^{\alpha} P\widehat{\mu}(m) = 0
\]
for every $P \in \mathcal{P}$ and every multi-index $\alpha \in \N^n$, where $n$ is the number of variables of $P$. We denote by $\mathcal{P}_{\widehat{I}, m}$ the set of all polynomial differential operators that annihilate $\widehat{I}$ at $m$. Conversely, let $\widehat{I}_{\mathcal{P}, m}$ be the set of all Fourier transforms that are annihilated by $\mathcal{P}$ at $m$. This set $\widehat{I}_{\mathcal{P}, m}$ forms a closed ideal.

The ideal $\widehat{I}$ is called \emph{localizable} (see \cite{Sze23}) if
\[
\widehat{I} = \bigcap_m \widehat{I}_{\mathcal{P}_{\widehat{I}, m}, m}.
\]
Equivalently, $\widehat{I}$ is \emph{non-localizable} if there exists a function $\widehat{\nu} \notin \widehat{I}$ such that $\widehat{\nu}$ is annihilated by $\mathcal{P}_{\widehat{I}, m}$ at every $m$.

%\vskip.2cm

Note that if $\mathcal{P} \ne \{0\}$ annihilates $\widehat{I}$ at $m$, then $\widehat{\mu}(m) = 0$ for every $\widehat{\mu}$ in $\widehat{I}$. Hence, every localizable ideal has a root. If $m$ is not a root of $\widehat{I}$, then $\mathcal{P}_{\widehat{I}, m} = \{0\}$, and thus $\widehat{I}_{\mathcal{P}_{\widehat{I}, m}, m} = \mathcal{A}(G)$. Therefore, the only localizable ideal without any root is $\mathcal{A}(G)$ itself.

%\vskip.2cm

The fundamental result concerning localizability is the following (see \cite[Theorem 4]{Sze23}):

\begin{Thm}
	An ideal $\widehat{I}$ is localizable if and only if $\Ann I$ is synthesizable.
\end{Thm}

\section{Roots of principal ideals on $\R^n$}

We recall some basic facts on localization (see \cite{Sze23}). Let $\widehat{I}$ be a principal ideal in $\mathcal{A}(\R^n)$ generated by the Fourier transform $\widehat{\mu}$. We show that if $\widehat{I}$ is proper, i.e., $\widehat{I} \ne \mathcal{A}(\R^n)$, then it has a root.

Let $K$ be the support of $\mu$. For each finite subset $F \subseteq K$, let $\mu_F$ denote the restriction of $\mu$ to the subgroup $G_F$ of $\R^n$ generated by $F$. More precisely, we define
\[
\langle \mu_F, f \rangle = \int f \cdot \chi_F\, d\mu,
\]
where $\chi_F$ denotes the characteristic function of the set $F$. The group $G_F$ is a (continuous) homomorphic image $\Phi(\Z^k)$ of $\Z^k$ for some positive integer $k$; in fact, it can be identified with a factor group of $\Z^k$. It follows that the exponentials on $G_F$ can be identified with those exponentials on $\Z^k$ that are constant on the cosets of the kernel of $\Phi$.

On the other hand, the Fourier transform of $\mu_F$ is a Laurent polynomial of the form $\widehat{\mu}_F = \frac{p}{q}$, where $p$ and $q$ are polynomials. Hence, the set $Z_F$ of roots of $\widehat{\mu}_F$ coincides with the zero set of the polynomial $p$. Consequently, $Z_F$ is an algebraic variety. It is nonempty, as spectral synthesis holds on finitely generated abelian groups. Moreover, $Z_F$ is compact in the Zariski topology.

Now consider the family of all sets of the form $Z_F$, where $F$ is a finite subset of the support of $\mu$. It is straightforward to verify that this family satisfies the finite intersection property. Since the family consists of Zariski-compact sets, we infer that $\bigcap_F Z_F \ne \emptyset$. Clearly, the elements of this intersection are roots of the Fourier transform of $\mu$, that is, of the principal ideal $\widehat{I}$.

%\vskip.2cm

We define a topological abelian group $G$ to be a \emph{PIT-group} if the Principal Ideal Theorem holds on $G$: that is, if the annihilator of a variety on $G$ is a principal ideal, then the variety is synthesizable—meaning the exponential monomials span a dense subspace in the variety.

Recall that an \emph{exponential monomial} on $G$ is a complex-valued function of the form
\[
x \mapsto P(a_1(x), a_2(x), \dots, a_k(x))\, m(x),
\]
where $P$ is a complex polynomial in $k$ variables, $a_1, a_2, \dots, a_k$ are additive functions, and $m$ is an exponential. Linear combinations of exponential monomials are called \emph{exponential polynomials}, and if $m = 1$, they are called \emph{polynomials}.

\begin{Thm}
	Every continuous homomorphic image of a PIT-group is a PIT-group.
\end{Thm}

\begin{proof}
	Let $G$ be a PIT-group, and let $\Phi: G \to H$ be a surjective continuous homomorphism. Consider the variety $W$ on $H$ defined as the solution space of the equation $\mu * f = 0$, where $\mu$ is a compactly supported complex Borel measure on $H$. Let $I = \Ann W$, and we aim to show that the ideal $\widehat{I}$ is localizable.
	
	Let $V$ be the set of all functions of the form $f \circ \Phi$ with $f \in W$. It is easy to see that $V$ is a variety on $G$. Indeed, $V$ is a linear space and translation invariant:
	\[
	(\alpha f + \beta g) \circ \Phi = \alpha (f \circ \Phi) + \beta (g \circ \Phi), \quad \tau_y(f \circ \Phi)(x) = \tau_{\Phi(y)} f(\Phi(x)).
	\]
	Moreover, $V$ is closed: if the net $(f_i \circ \Phi)$ converges to $F$ on $G$ with $f_i \in W$, then $F = \varphi \circ \Phi$ for some $\varphi \in W$, hence $F \in V$.
	
	Define $\mu_{\Phi}$ on $V$ by
	\[
	\langle \mu_{\Phi}, f \circ \Phi \rangle = \langle \mu, f \rangle.
	\]
	Then $\mu_{\Phi}$ is a linear functional on $V$, and by the Hahn–Banach Theorem, it extends to a linear functional on $\mathcal{C}(G)$—denoted by the same symbol. Thus, $\mu_{\Phi} \in \mathcal{M}_c(G)$, and $f \circ \Phi \in V$ if and only if $\mu_{\Phi} * (f \circ \Phi) = 0$. In other words, the annihilator $J = \Ann V$ is the principal ideal generated by $\mu_{\Phi}$.
	
	Since $G$ is a PIT-group, the exponential polynomials in $V$ form a dense subspace. Let $f \circ \Phi$ be an exponential polynomial in $V$. Then $f$ is also an exponential polynomial. By \cite[Theorem 3]{MR2680008}, it suffices to show that the variety of $f$ is finite-dimensional, which follows from the same property of $f \circ \Phi$.
	
	If $f \in W$, then $f \circ \Phi \in V$, and there exists a net $(\varphi_i \circ \Phi)$ of exponential polynomials in $V$ converging to $f \circ \Phi$. Hence, the net $(\varphi_i)$ consists of exponential polynomials in $W$ and converges to $f$. This completes the proof.
\end{proof}

Recently, in \cite[Theorem 1]{MR4789359}, we proved that spectral synthesis holds on a locally compact abelian group $G$ if and only if it holds on $G/B$, where $B$ is the closed subgroup of $G$ generated by its compact elements. With a slight modification of the proof, we obtain the following:

\begin{Thm}\label{comp}
	A locally compact abelian group $G$ is a PIT-group if and only if $G/B$ is a PIT-group, where $B$ is the closed subgroup of $G$ generated by its compact elements.
\end{Thm}

Using these theorems, we can now prove our main result.

\begin{Thm}
	Every locally compact abelian group is a PIT-group.
\end{Thm}

\begin{proof}
	We first show that $\R^n$ is a PIT-group. Let $\widehat{I}$ be a proper principal ideal in $\mathcal{A}(\R^n)$, and let $z_0$ be a root of $\widehat{I}$ - as shown above, such a root always exists. Let $\widehat{M}_0$ be the corresponding maximal ideal in $\mathcal{A}(\R^n)$, and set $S_0 = \mathcal{A}(\R^n) \setminus \widehat{M}_0$. Define the localization
	\[
	R_0 = S_0^{-1} \mathcal{A}(\R^n).
	\]
	Then $R_0$ is a local ring with unique maximal ideal $S_0^{-1} \widehat{M}_0$. In fact, $R_0$ can be identified with the ring of germs of meromorphic functions at $z_0$, whose numerators lie in $\mathcal{A}(\R^n)$ and denominators in $S_0$. The ring $R_0$ is Noetherian. Let $S_0^{-1} \widehat{I}$ be the ideal in $R_0$ corresponding to $\widehat{I}$. Then the residue ring $R_0 / S_0^{-1} \widehat{I}$ is also Noetherian, with unique maximal ideal $S_0^{-1}(\widehat{I} + \widehat{M}_0)$.
	
Let $P$ be a differential operator that annihilates $\widehat{I}$ at $z_0$. Then $P$ also annihilates $S_0^{-1} \widehat{I}$ at $z_0$. The converse is true as well: the same differential operators annihilate both $\widehat{I}$ and $S_0^{-1} \widehat{I}$ at $z_0$. This follows directly from the representation of elements of $R_0$ as meromorphic functions.

%\vskip.2cm

By the Noetherian property and Krull's Intersection Theorem, we have
\[
S_0^{-1} \widehat{I} = \bigcap_n S_0^{-1}(\widehat{I} + \widehat{M}_0)^n = \bigcap_n \left( S_0^{-1} \widehat{I} + S_0^{-1} \widehat{M}_0^n \right).
\]
It follows that the ideal $S_0^{-1} \widehat{I}$ is synthesizable (see \cite[Theorem 17]{MR3502634}), and hence it is localizable .

Now assume that $\widehat{\nu}$ is annihilated at $z_0$ by every differential operator that annihilates $\widehat{I}$ at $z_0$. Then $\widehat{\nu}/1$ is annihilated at $z_0$ by every differential operator that annihilates $S_0^{-1} \widehat{I}$ at $z_0$. By the localizability of $S_0^{-1} \widehat{I}$, we infer that $\widehat{\nu}/1$ is  in $S_0^{-1} \widehat{I}$, and consequently $\widehat{\nu}$ belongs to $\widehat{I}$. This argument applies for each root $z_0$ of the ideal $\widehat{I}$, which implies that $\widehat{I}$ is localizable. Therefore, by the results in \cite{Sze23}, it is synthesizable. This completes the proof.

%\vskip.2cm

To consider the general case, we may assume that $G$ is compactly generated, and hence it has the form $G = \R^n \times \Z^k \times C$, where $n, k$ are nonnegative integers and $C$ is a compact abelian group. By Theorem~\ref{comp}, the problem reduces to the case where $G = \R^n \times \Z^k$. Every compactly supported measure on $\R^n \times \Z^k$ can be regarded as a compactly supported measure on $\R^{n+k}$, and every exponential monomial on $\R^n \times \Z^k$ extends to an exponential monomial on $\R^{n+k}$. It follows that every compactly generated locally compact abelian group is a PIT-group. This completes the proof.
\end{proof}

\begin{Cor}
	Let $G$ be a locally compact abelian group. For each compactly supported complex measure $\mu$ on $G$ the exponential monomials span a dense subspace in the solution space of the functional equation $\mu*f=0$.
\end{Cor}
	
\section{Statements and Declarations}
Data sharing is not applicable to this article, as no datasets were generated or analyzed. The author declares that there are no financial or non-financial interests directly or indirectly related to the work submitted for publication.

\end{document}